%% file: PEVA_arxiv_revision1.tex
\documentclass[letterpaper,11pt]{article}
\input{my_header}

\input{arxiv_header}

\begin{document}

\include{arxiv_author}

\maketitle

\section{Introduction}
\label{intro}

Traditional queueing systems involve multiple queues interacting with each other. The analysis of the queueing mechanism and related load balancing questions are tractable owing to the product-form stationary distribution. Most earlier works on queues distributed spatially involve customers arriving at random locations in a Euclidean space and a server travelling to serve them. However, most queueing networks in the present day involve servers that are distributed in space and arrivals deciding among multiple servers. This work addresses the problem of load distribution in such networks where the servers are spatially distributed and the arrivals decide to join a server based on proximity.

As a motivating example, consider the rapidly growing electric vehicles (EV) industry. With increased adaptation of EVs,  the charging infrastructure is also being scaled up.  However, physical and financial constraints put a cap on the number of charging stations that can be deployed. In such a scenario, it is natural to expect EV users to adopt strategies in order to minimize their waiting times at charging stations. For example, a user on arriving at a charging station and finding it to be occupied, might decide to travel to a farther station in the hope that it would be empty. The user strategy affects the load that is perceived at the servers in the long run. Some of the servers get overloaded which might degrade the performance of the system on the whole. An understanding of the fraction of overloaded servers helps in optimal resource allocation preventing such degradation. Additionally, monitoring the fraction of overloaded servers can be used to incentivize customers to change their strategies thus increasing the durability of the entire system. Similar considerations arise in other practical networks involving queues, like supermarkets, airports etc.

Motivated by such applications, in this work we consider a set of $N$ servers that are deployed in a Euclidean space with queues at each of them modelled as stationary arrival processes. In the context of EVs, the servers are the charging stations, and a stationary arrival process models the EV users arriving at a charging station. EV charging stations have been modelled as a Poisson point process in two dimensional space in several previous works (see, for e.g., \citep{dong_electric_2019,ren_novel_2020,atat_stochastic_2020}). However, the problem that authors address in these works pertain to optimal placement of charging stations in the underlying space. 

Another line of work that considers queues on spaces are polling systems. One of the earlier works in this domain \citep{altman_queueing_1994} considers multiple queues in a convex space with a single server moving across to serve them. Several later works, such as \citep{kavitha_queuing_2009,celik_dynamic_2010}, study vehicle routes and delay in such systems. In all of these (and some related) works there is a single server and there is no interaction between different queues.

In contrast, in the present work we consider queues where customers from one queue probabilistically move to a queue which is located close to it in the underlying space. Customers changing queues has been referred to as \emph{jockeying} in the queueing theory terminology.
There has been considerable work on jockeying in queues in the past \citep{koenigsberg_jockeying_1966,dudin_analysis_2023,lin_optimal_2022}. The focus in these works is predominantly to analyse the steady-state distribution or find expected line-lengths or delays. However, they primarily consider two server systems with no spatial component in the problem.

In \citep{stacey_greedy_2011}, the authors consider arrivals occurring on the two-dimensional torus, with multiple servers following a greedy strategy to serve the customers by travelling minimally. They find that such a strategy results in servers coalescing making the system inefficient. In contrast, in our work the servers are fixed whereas the arrivals follow nearest neighbour strategies and hence the two works are not directly comparable.

The closest work to ours is \citep{panigrahy2022analysis} where the authors consider load-balancing policies for servers distributed in space. Specifically, $m$ servers are distributed in a two-dimensional Euclidean space and $n$ customers arrive distributed uniformly in space. The customers choose the shorter queue among the server they arrived at and a second server sampled uniformly at random from $k$ geographically nearest servers  of the server where the customer arrived. Specifically, the authors observe  that such a policy shows a better performance with respect to the maximum load. While these observations are interesting in their own right, the performance metrics of interest and the model setting are considerably different from the current work. Instead of a setting where there are fixed number of arrivals, we consider stationary arrival processes into each of the servers with customers following nearest neighbour mobility strategies. Additionally, instead of the maximum load and the allocation distance, our interest is in the distribution of the number of overloaded servers in the system. These differences render the two works incomparable. However, it should be remarked that there is considerable scope for combining ideas from both the works.

More precisely, in the current work, $N$ servers are deployed in a Euclidean space and customers arriving into a queue decide whether to stay or to move to a queue nearest to them with a prescribed probability. We call such strategies as nearest neighbour shift (NNS) strategies. The metric of interest is the expected fraction of overloaded servers in the system which we characterize for users following an NNS strategy. 

We begin by describing our system model and state the main results in Section \ref{sec:model}. Section \ref{sec:prelims} provides the required background on nearest neighbour graphs which will be used to characterize the overloaded servers for the NNS strategies. This is followed by the analysis of NNS strategies for general dimension in Section \ref{sec:2d} and additional results for one dimension in Section \ref{sec:1d}. Section \ref{sec:simu} provides numerical simulations justifying our results and Section \ref{sec:conc_fw} highlights some future directions.

\section{System model and main results}
\label{sec:model}

\subsection{Model description}

Consider $N$ service stations with locations $\{X_i\}_{i=1}^N$ uniformly and independently distributed within $[0,1]^d$, $d \ge 1$, equipped with the standard Euclidean metric. Each service station is associated with a stationary arrival process with an \emph{exogenous arrival rate} of $\lambda$, a server operating at a service rate $\mu$, and a queue of infinite capacity.  We use the terms queue~$i$ and server~$i$ to refer to the service station located at $X_i$.
Customers arriving at a queue can adopt different strategies. In this work, we consider all customers to be identical who make independent decisions.  Our interest is in \emph{nearest neighbour shift} (NNS) strategies which we describe next. 

\begin{defn}
Under the \emph{$(k,p)$-NNS strategy} with activity range $k \ge 1$ and a shift rate $p \in [0,1]$, customers arriving at queue~$i$ join queue~$i$ with probability $1-p$, and join queue $j \in \cN_k(i)$ with probability $\frac{p}{k}$, where $\cN_k(i)$ denotes the set of $k$ nearest neighbours of $i$.

\label{defn:kpnns}
\end{defn}

We emphasize the difference between a customer \emph{arriving} at a queue and \emph{joining} a queue.  Whereas each queue~$i$ has exogenous arrival rate $\lambda$, the rate at which customers join queue~$i$ may be smaller or larger than $\lambda$.

\begin{defn}
The \emph{effective arrival rate} $\leff(i)$ at a server $i$ is the rate at which customers join queue~$i$. 
\label{defn:leff}
\end{defn}

Our goal is to characterize the expected fraction of overloaded servers in the system. Because each queue is processed at service rate $\mu$, 
the load at queue~$i$ is equal to $\rho_i = \frac{\leff(i)}{\mu}$, and 
the fraction of overloaded servers (with $\rho_i > 1$) can be written as
\begin{equation}
 \cO_N
 = \frac{1}{N} \left|\left\{i: \leff(i) > \mu \right\}\right|,
 \label{eq:overload}
\end{equation}
where $|\cdot|$ denotes the cardinality. Note that $\cO_N$ is a random variable where the randomness is due to the spatial distribution of the servers.

\subsection{Main results for general dimension}
\label{sec:MainResultsForGeneralDimension}

Our main results characterize the fraction of overloaded servers for the $(k,p)$-NNS strategy for $N$ servers randomly distributed within $[0,1]^d$. 
The results are described in terms of a geometric constant
$\alpha_d$, defined as the maximum number of points that can be placed on
the unit sphere of $\R^d$ so that the distance between all distinct points is
strictly larger than one.  It is known \citep{bahadirNumberWeaklyConnected2021} that $\alpha_2 = 5$ and $\alpha_3 = 12$,
and that $\alpha_d$ is bounded from above by the kissing number in $\R^d$.

\begin{thm}
Consider $N$ servers with service rate $\mu$ that are uniformly distributed in $[0,1]^d$ and receive exogenous arrivals at rate $\lambda$ following the $(k,p)$-NNS strategy.
\begin{enumerate}[label=(\roman*)]
\item 
If $\frac{1}{1-p + \alpha_d p} < \frac{\lambda}{\mu} \le 1$, then there exist constants
$q_{d,k,n}$ and $\sigma_{d,k}$
such that for $N \to \infty$, the fraction of overloaded servers satisfies 
\[
 \cO_N \asto
 \sum_{n=\floor{\theta}+1}^{\alpha_d k} q_{d,k,n}
 \quad\text{and}\quad
 \sqrt{N} \big(\cO_N - \E \cO_N \big) \dto \cN(0,\sigma_{d,k}^2)
\]
where $\theta = k + \frac{k}{p} \left( \frac{\mu}{\lambda} -1 \right)$.

\item If $\frac{\lambda}{\mu} \le \frac{1}{1-p + \alpha_d p}$,
then $\cO_N = 0$ a.s.\ for all $N$.
\end{enumerate}

\label{thm:main2d}
\end{thm}

Theorem~\ref{thm:main2d} confirms that the number of overloaded servers has a normal distribution in the regime when $N \to \infty$. This is illustrated in our numerical results in Section~\ref{sec:simu}.
The constants $q_{d,k,n}, \sigma_{d,k}$ are discussed in Section~\ref{sec:knn_graph}.
Note that the almost sure ($\asto$) and distributional ($\dto$) limits above refer to the randomness induced by the spatial distribution of the servers.

\subsection{Main results for one dimension}

As a corollary of Theorem~\ref{thm:main2d}, we obtain the following explicit result 
for dimension $d=1$ in which the arrivals follow the $(1,p)$-NNS strategy.

\begin{corr}
Consider $N$ servers with service rate $\mu$ that are uniformly distributed in $[0,1]$ and receive exogenous arrivals at rate $\lambda$ following the $(1,p)$-NNS strategy.
\begin{enumerate}[label=(\roman*)]
\item 
If $\frac{1}{1+p} < \frac{\lambda}{\mu} \le 1$, then the fraction of overloaded servers satisfies
\[
 \cO_N \asto \frac14
 \quad \text{and} \quad
 \sqrt{N}\big(\cO_N-\E\cO_N\big)
 \dto \cN\Big(0,\frac{19}{240}\Big)
 \quad\text{as $N \to \infty$}.
\]

\item If $\frac{\lambda}{\mu} \le \frac{1}{1+p}$,
then $\cO_N = 0$ a.s.\ for all $N$.
\end{enumerate}
\label{thm:convergence_1d}
\end{corr}

In one dimension, additional explicit results can be obtained.
In fact, for $d=1$ and $k=1$, we obtain the exact expression for $\E\cO_N$ for all $N\ge 1$. Specifically, for the $(1,p)$-NNS strategy, we find that a quarter of the servers get overloaded  which is stated in the following theorem. 

\begin{thm}
Consider $N \ge 4$ servers distributed uniformly in $[0,1]$ with exogenous arrivals of rate $\lambda$ following the $(1,p)$-NNS strategy and having a service rate $\mu$. If $\frac{1}{1+p} < \frac{\lambda}{\mu} \le 1$, then the expected fraction of overloaded servers is $\E\cO_N = \frac{1}{4}.$ 
\label{thm:main1d}
\end{thm}

In the process of proving this theorem, we also obtain the expected fraction of servers with all possible loads. This is detailed in Section \ref{sec:1d}.

Often, EV users tend to have a pre-determined destination while joining a highway and choose either to go left (with probability $\ell$) or right (with probability $r$) and proceed to the nearest server in that direction (or stay in the same queue with probability $1-r-\ell$). Customers at the left (right) boundary, choose the closest neighbour on the right (resp., left) with probability $r$ (resp., $\ell$) or stay in the same queue with the remaining probability. For such a \emph{left-right nearest neighbour shift ($\ell r$-NNS)} strategy, the following theorem asserts that there are no overloaded servers.

\begin{thm}
For $N$ servers distributed uniformly in $[0,1]$ with stationary exogenous arrivals of intensity $\lambda \le \mu$ following the $\ell r$-NNS strategy with parameters $r$ and $\ell$, the fraction of overloaded servers equals $\cO_N = 0$ a.s.
\label{thm:lrnns}
\end{thm}

\section{Nearest neighbour graphs}
\label{sec:prelims}
In this section, we provide some background knowledge required in order to prove the main results. The proof of Theorem \ref{thm:main1d} relies on the distribution of order statistics which is detailed in Section \ref{sec:order_stats}. Section \ref{sec:knn_graph} provides some results on $k$-NN graphs that are required to prove Theorem \ref{thm:main2d}.

\subsection{Order statistics}
\label{sec:order_stats}

The primary tool that we will use in the $1$D case are the order statistics of the server locations which we define below.
Let $X_1, \dots, X_N$ be iid random variables from a distribution $F(\cdot)$ with density $f(\cdot)$ supported on $[0,1]$. When arranged in the order of magnitude and then written as $X_{(1)} \le X_{(2)} \le \cdots \le X_{(N)}$, the random variable $X_{(i)}$ is called the $i$-th \emph{order statistic}, and together they are referred to as the \emph{order statistics} of $X_1, \dots, X_N$.

The following are some facts about order statistics that will be used in our analysis (see e.g., \citep{david2004order,ross2010first}). These are stated for $f(x) =  \1\{x\in [0,1]\}$ and $F(x) = \int_{0}^xf(x)dx$ since in our case the servers are distributed uniformly in $[0,1]$. 
\begin{itemize}[leftmargin=*]
\item \textbf{Fact 1:} Conditioned on $X_{(i)} = a$, the random variables $X_{(1)}, \dots, X_{(i-1)}$ are the order statistics of
\begin{equation}
 Y_1, \dots, Y_{i-1}
 \stackrel{\text{iid}}{\sim}
 \text{Unif}\left([0,a]\right),
 \label{eq:cond_order_stat_lower}
\end{equation}
and $X_{(i+1)}, \dots, X_{(N)}$ are the order statistics of
\begin{equation}
 Y_{i+1}, \dots, Y_N
 \stackrel{\text{iid}}{\sim}
 \text{Unif} \left([a,1]\right).
 \label{eq:cond_order_stat_upper}
\end{equation}
Moreover, $X_{(1)}$, \dots, $X_{(i-1)}$ and
$X_{(i+1)}, \dots, X_{(N)}$ are conditionally independent given $X_{(i)}$.

\item \textbf{Fact 2:} The joint density of the $i$-th and the $j$-th order  statistic for $i<j$ is given by
\begin{align}
 f_{i,j}(r,s)
 &= \frac{N!}{(i-1)!(j-i-1)!(N-j)!}F(r)^{i-1}f(r) \nonumber \\
 &\ \hspace{2cm} \times \left[F(s)-F(r)\right]^{j-i-1}f(s)[1-F(s)]^{N-j}.
 \label{eq:dist_order_stats}
\end{align}
\end{itemize}

\subsection{$k$-NN graph}
\label{sec:knn_graph}

To characterize the $(k,p)$-NNS strategy for general dimension, we map the problem onto a graph and use results from stochastic geometry.

\begin{defn}
The \emph{$k$-NN graph} $G^N_{d,k}=(V,E)$ is a directed graph on the vertex set $V=\{X_1,\cdots,X_N\}$ with $X_i \in [0,1]^d$,
and the edge set is constructed in the following way:
add a directed edge
$i\to j$, if node $j$ is one of the $k$-nearest neighbours of node $i$.     
\label{defn:knn_graph}
\end{defn}

In this work, the vertex set $V$ for the $k$-NN graph is taken to be $N$ nodes distributed uniformly within $[0,1]^d$. The \emph{in-degree} of a vertex $i$ is the number of incoming edges to $i$, i.e., $|\{j: j\to i\}|$ and the \emph{out-degree} is the number of outgoing edges from $i$ i.e., $ |\{j: i\to j\}|$. Some properties of the $k$-NN graph (see e.g., \citep{bahadirNumberWeaklyConnected2021}) are listed below:
\begin{enumerate}
\item Every node $i \in G^N_{d,k}$ has out-degree $k$.
\item \label{kissing}
Every node $i \in G^N_{d,k}$ has in-degree at most $\alpha_d k$,
where $\alpha_d$ is the geometric constant defined in Section~\ref{sec:MainResultsForGeneralDimension}.

\end{enumerate}
A directed graph $H_1 = (V_1,E_1)$ is said to be a \emph{subgraph} of $H_2 = (V_2,E_2)$, denoted $H_1 \subseteq H_2$, if $V_1 \subseteq V_2$ and $E_1 \subseteq E_2$. Two directed graphs $H_1 $ and $H_2 $ are said to be \emph{isomorphic} to each other if there exists a bijection $f:V_1 \to V_2$ such that $(i,j) \in E_1$ if and only if $(f(i),f(j)) \in E_2$. A directed graph is said to be \emph{weakly connected} if the graph obtained on replacing the directed edges with undirected edges is connected. A subgraph that is weakly connected is referred to as a \emph{component}.

Let $I_D^N$ be the number of directed subgraphs of $G_{d,k}^N$ that are isomorphic to
a particular directed graph $D$.
The convergence of these subgraph counts is characterized in
\citep{bahadirNumberWeaklyConnected2021} which is reproduced below.
Here $\cN(\mu,\sigma^2)$ denotes a normal distribution with mean $\mu$ and variance $\sigma^2$.

\begin{thm}[{\citep[Theorem 2.1]{bahadirNumberWeaklyConnected2021}}]
For any weakly connected directed graphs $D_1,\cdots, D_m$ and any real numbers $a_1,\cdots,a_m$
there exist constants $\xi$ and $\sigma$ such that
$I^N= a_1 I_{D_1}^N+a_2 I_{D_2}^N +\cdots + a_m I_{D_m}^N$ satisfies

\begin{align*}
 \frac{I^N}{N} \asto \xi
 \qquad\text{and}\qquad
 \frac{I^N - \E \big[I^N\big]}{\sqrt{N}} \dto \cN(0,\sigma^2)
 \qquad \text{as $N \to \infty$}.
\end{align*}
\label{thm:iso_conv}
\end{thm}

The subgraph of our interest is the directed star graph $K_i= (V_i,E_i)$ with $V_i = \{1,2,\cdots,i+1\}$ and $E_i = \{(1,i+1),(2,i+1),\cdots,(i,i+1)\}$ as shown in Fig. \ref{fig:star}. The directed star graph captures the in-degree of a node which is key to our analysis.
\begin{figure}
\centering
\includegraphics[width=0.35\linewidth]{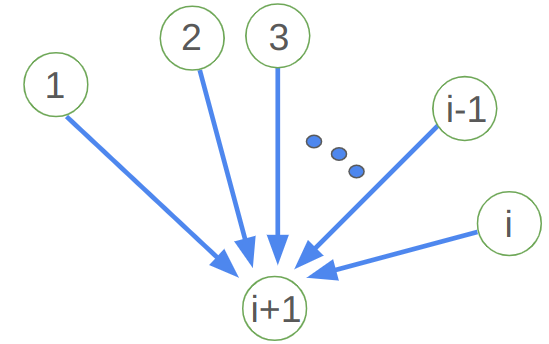}
\caption{The star graph $K_i$ formed by a node with in-degree $i$.}
\label{fig:star}
\end{figure}
Let $Q_{d,k,j}^N$ denote the number of nodes whose in-degree is $j$ in $G_{d,k}^N$.
The following proposition relates the quantities $Q_{d,k,j}^N$ and $I_{K_i}^N$. We include a proof of it here since it was omitted in \citep{bahadirNumberWeaklyConnected2021}.
\begin{prop}
For $N>\alpha_d k +1$,
\begin{equation}
 Q_{d,k,j}^N
 = \sum_{i=j}^{\alpha_d k}(-1)^{i-j}\binom{i}{j}I_{K_i}^N.
\end{equation}
\label{prop:indeg_iso}
\end{prop}

\begin{proof}
Since $K_i \subseteq K_{i+1}$, the number of copies of $K_i$ in $G^N_{d,k}$ can be written as 
$$I_{K_i}^N = Q_{d,k,i}^N+\binom{i+1}{i}Q_{d,k,i+1}^N +\cdots + \binom{\alpha_d k}{i}Q_{d,k,\alpha_d k}^N.$$
Stacking these equations into a matrix form for $1\le i\le \alpha k$, we obtain
\begin{equation}
\begin{bmatrix}
I_{K_1}^N \\
I_{K_2}^N \\
\vdots \\
I_{K_{\alpha_d k-1}}^N\\
I_{K_{\alpha_d k}}^N
\end{bmatrix} = 
\begin{bmatrix}
1 & \binom{2}{1} &\cdots&\binom{\alpha_d k-1 }{1}&\binom{\alpha_d k }{1}\\
0 & 1  & \cdots&\binom{\alpha_d k-1 }{2}&\binom{\alpha_d k }{2}\\
\vdots &\vdots & \ddots &\vdots&\vdots \\
0 &0 & \cdots &1&\binom{\alpha_d k }{\alpha_d k -1}\\
0 &0 & \cdots &0 &1\\
\end{bmatrix}
\begin{bmatrix}
Q_{d,k,1}^N \\
Q_{d,k,2}^N \\
\vdots \\     
Q_{d,k,\alpha_d k-1}^N\\
Q_{d,k,\alpha_d k+1}^N
\end{bmatrix}.
\end{equation}
The system of linear equations can be solved for the variables $Q_{d,k,\cdot}^N$ to obtain the statement of the proposition.
\end{proof}
As a corollary of Proposition \ref{prop:indeg_iso} and Theorem \ref{thm:iso_conv}, we obtain the convergence of the random variables $Q_{d,k,j}^N$ as stated below.
\begin{corr}
For $0\le j\le \alpha_d k$, there exist constants $q_{d,k,j}$ and $ \sigma_{d,k,j}$
such that as $N \to \infty$,
\begin{align*}
 \frac{Q_{d,k,j}^N}{N} \asto q_{d,k,j}
 \quad\text{and}\quad
 \frac{Q_{d,k,j}^N - \E \big[Q_{d,k,j}^N\big]}{\sqrt{N}} \dto \cN(0,\sigma_{d,k,j}^2).
\end{align*}

\label{thm:indeg_frac}
\end{corr}
The value of the constant $q_{d,k,j}$ is hard to compute for general $k$. For the case of $k=1$, \citep{newman1983nearest} shows that
\begin{equation}
q_{d,1,j} = \frac{1}{j!}\sum_{i=0}^{\alpha_d -j} \frac{(-1)^i}{i!}C_{i+j}^d,
\label{eq:qjdk}
\end{equation}
where 
\[
C^d_{k}
\weq \int_{A_k^d} e^{- \text{vol} \big( B(u_1, \|u_1\|) \cup \cdots \cup B(u_k, \|u_k\|) \big)} du_1 \dots du_k,
\]
with $B(u,r)$ denoting the ball of radius $r$ centered at $u$,
\[
A_k^d
\weq \Big\{ (u_1,\dots, u_k) \in (\R^d)^k \colon \|u_i\| < \|u_j-u_i\| \ \text{for all $i \ne j$} \Big\},
\]
and $\text{vol}(B)$ denoting the Lebesgue measure of set $B$.

\section{Proof for general dimension}
\label{sec:2d}

In this section we present the proof of Theorem~\ref{thm:main2d}
characterising the fraction of overloaded service stations
for the $(k,p)$-NNS strategy in $[0,1]^d$.

\begin{proof}[Proof of Theorem \ref{thm:main2d}]

Assume that $\lambda < \mu$ and $0 < p \le 1$.
We denote by $\din(i)$ and $\dout(i)$ the in-degree and the out-degree
of node~$i$ in the $k$-NN graph of the servers.
Because the server locations are independently and uniformly distributed
in $[0,1]^d$, we see that all inter-server distances are unique
with probability one.  As a consequence, $0 \le \din(i) \le \alpha_d k$
and $\dout(i) = k$ for all $i$ almost surely, as explained in Section~\ref{sec:knn_graph}.
Recall that customers arriving at location $i$ join queue~$i$ with
probability $1-p$, and customers arriving at location $j$ with $i\in \cN_k(j)$
join queue~$i$ with probability $\frac{p}{k}$.
Because the arrival processes are stationary and mutually independent with intensity $\lambda$, it follows that the net rate of customers joining queue~$i$ equals $\leff(i) = \lambda (1-p) + \lambda \din(i) \frac{p}{k}$.
We conclude queue~$i$ is overloaded if and only if
\[
 \lambda (1-p) + \lambda \din(i) \frac{p}{k} > \mu,
\]
or equivalently, $\din(i) > \theta$ where
$
 \theta = k + \frac{k}{p} \left( \frac{\mu}{\lambda} -1 \right).
$
As a consequence, the fraction of overloaded queues in the system can be written as
\[
 \cO_N = \frac{1}{N} \sum_{n: n > \theta} Q^N_{d,k,n},
\]
where $Q^N_{d,k,n}$ equals the number of servers with in-degree $n$.
Because $\din(i) \le \alpha_d k$, we note that
$Q^N_{d,k,n} = 0$ for $n > \alpha_d k$ almost surely. Therefore,
with probability one,
\begin{equation}
 \label{eq:OverloadFractionSum}
 \cO_N
 =
 \begin{cases}     
  \frac{1}{N} \sum_{n = \floor{\theta}+1}^{\alpha_d k} Q^N_{d,k,n},
  &\qquad \theta < \alpha_d k, \\
  0, &\qquad \theta \ge \alpha_d k.
 \end{cases}
\end{equation}
We also note that $\theta \ge \alpha_d k$ if and only if
$\frac{\lambda}{\mu} \le \frac{1}{1-p+\alpha_d p}$.
Hence $\cO_N = 0$ a.s.\ for $\frac{\lambda}{\mu} \le \frac{1}{1-p+\alpha_d p}$.

Let us now assume that $\frac{1}{1-p+\alpha_d p} < \frac{\lambda}{\mu} \le 1$.
Then $\theta < \alpha_d k$.
Corollary~\ref{thm:indeg_frac} then implies that
\[
 \cO_N
 \asto \sum_{n=\floor{\theta}+1}^{\alpha_d k} q_{d,k,n}
 \qquad \text{as $N \to \infty$}.
\]
Moreover, Proposition~\ref{prop:indeg_iso} implies that
\[
 \sum_{n=\floor{\theta}+1}^{\alpha_d k}Q^N_{d,k,n}
 = \sum_{n=\floor{\theta}+1}^{\alpha_d k}
   \sum_{m=n}^{\alpha_d k}(-1)^{m-n}\binom{m}{n}I_{K_m}^N
 = \sum_{m=\floor{\theta}+1}^{\alpha_d k}
 \Bigg(\sum_{n=\floor{\theta}+1}^{m}(-1)^{m-n}\binom{m}{n}\Bigg) I_{K_m}^N.
\]
We conclude that the fraction of overloaded servers can be written as
\[
 \cO_N
 \weq \frac{1}{N} \sum_{m=1}^{\alpha_d k} a_m I^N_{K_m},
\]
where
\[
 a_m
 =
 \begin{cases}
  \sum_{n=\floor{\theta}+1}^{m}(-1)^{m-n}\binom{m}{n}, &\qquad m \ge \floor{\theta}+1, \\
  0, &\qquad m \le \floor{\theta}.
 \end{cases}
\]
Theorem~\ref{thm:iso_conv} then asserts that
$\sqrt{N}\big(\cO_N-\E\cO_N\big)\dto \cN(0,\sigma^2)$
for some constant~$\sigma$.
\end{proof}

\section{Proofs for one dimension}
\label{sec:1d}

The case of dimension $d=1$ is interesting in its own right since it can be motivated by numerous applications. With reference to the examples discussed in the introduction, EV charging stations located on a highway and checkout counters at supermarkets correspond to queues on a $1$D space.

\subsection{In-degree frequencies in the one-dimensional 1-NN graph}

In this section we analyse the in-degree frequencies in the 1-NN graph $G_{1,1}^N$
in dimension $d=1$,
which will be useful in characterizing the overloaded servers in one dimension.
In $G_{1,1}^N$, a particular vertex can have in-degrees $\{0,1,2\}$. This is because a node $i$ could be the nearest neighbour to at most two vertices on either side of $i$. Recall that $Q^N_{d,k,n}$ equals the number of nodes with in-degree $n$ in $G_{d,k}^N$. The following lemma describes the expected in-degree frequencies
in $G_{1,1}^N$.

\begin{prop}
For any $N \ge 4$, the expected proportions of nodes with in-degree $s=0,1,2$ in the 1-NN graph on $[0,1]$ are given by 
\[
 \E \bigg[\frac{Q^N_{1,1,0}}{N}\bigg]=\frac{1}{4},\ \ \
 \E \bigg[\frac{Q^N_{1,1,1}}{N}\bigg] = \frac{1}{2}, \ \ \text{ and } \ \ 
 \E \bigg[\frac{Q^N_{1,1,2}}{N}\bigg] = \frac{1}{4}.
\]
\label{prop:frac_indeg}
\end{prop}

\begin{proof}
    Let $X_{(1)}, \cdots , X_{(N)}$ denote the order statistics of the locations of the $N$ servers. In the following discussion, \emph{node} $i$ will refer to the $i$-th order statistic. 
For $3\le i \le N-2$, node $i$ has no incoming edges in $G_{1,1}^N$ if and only if both the events $E_i^+ := \{X_{(i+2)} - X_{(i+1)} < X_{(i+1)} - X_{(i)}\}$ and $E_i^- := \{X_{(i-1)} - X_{(i-2)} < X_{(i)} - X_{(i-1)}\}$ occur simultaneously. We first compute the probability of the event $E_i^+ \cap E_i^-$. 

From Fact 1 of Section \ref{sec:order_stats}, conditioned on the value of $X_{(i)}=x$, the events $E_i^+$ and $ E_i^-$ are independent. Thus
\begin{align}
\P(E_i^+ \cap E_i^-)&
= \int_{0}^{1}\P(X_{(i+2)} >2X_{(i+1)} - x|X_{(i)}=x) \nonumber\\
& \qquad \qquad\P( X_{(i-2)}> 2X_{(i-1)} -x|X_{(i)}=x)f_{X_{(i)}}(x)dx,
\label{eq:zero_q}
\end{align}
where $f_{X_{(i)}}(\cdot)$ is the density of the $i$-th order statistic. We will now evaluate each of the probabilities in the above expression using \eqref{eq:cond_order_stat_lower}, \eqref{eq:cond_order_stat_upper} and \eqref{eq:dist_order_stats}. For this, we first write down the joint distribution of the $j$-th and the $(j+1)$-th order statistic using \eqref{eq:dist_order_stats} to be
\begin{equation}
f_{j,j+1}(r,s)= N!\ \frac{F(r)^{j-1}f(r)f(s)(1-F(s))^{N-j-1}}{(j-1)!(N-j-1)!}.
\label{eq:succ_order_stats}
\end{equation}
Using \eqref{eq:cond_order_stat_lower}, \eqref{eq:cond_order_stat_upper} with \eqref{eq:succ_order_stats} the required integral can be computed as follows: (see Fig. \ref{fig:integral})
\begin{figure}
\centering
\includegraphics[width=0.65\linewidth]{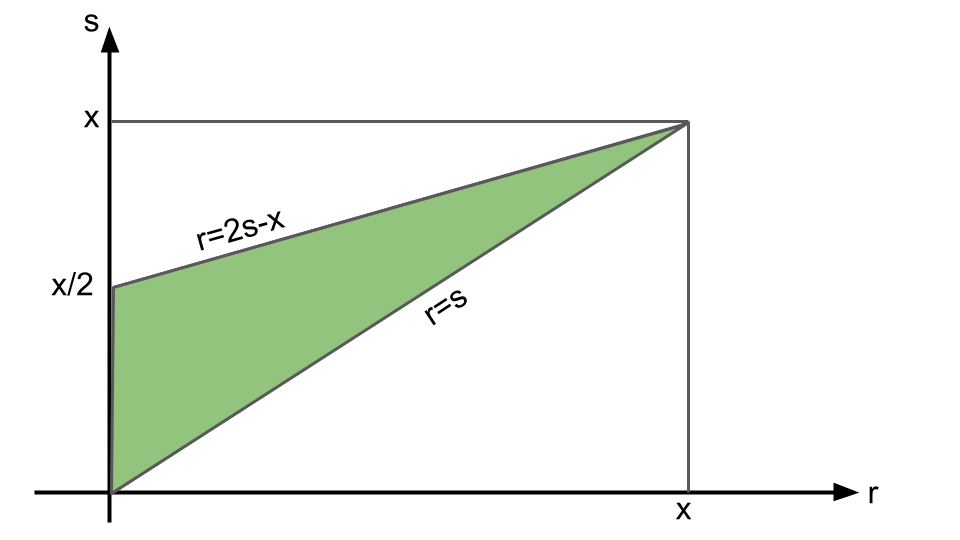}
\caption{Region of integration.}
\label{fig:integral}
\end{figure}
\begin{align*}
\P( X_{(i-2)}&> 2X_{(i-1)} -x|X_{(i)}=x)\\
&=\int_{0}^{x}\int_{0}^{s}\frac{(i-1)!}{(i-3)!}\left(\frac{r^{i-3}}{x^{i-1}}\right)\ dr ds -\int_{\frac{x}{2}}^{x}\int_{0}^{2s-x}\frac{(i-1)!}{(i-3)!}\left(\frac{r^{i-3}}{x^{i-1}}\right)\ dr ds\\
&=\int_{0}^{x}\frac{(i-1)(i-2)}{x^{i-1}}\left(\frac{s^{i-2}}{i-2}\right) ds -\int_{\frac{x}{2}}^{x}\frac{(i-1)(i-2)}{x^{i-1}}\left(\frac{(2s-x)^{i-2}}{i-2}\right) ds\\
&=\frac{i-1}{x^{i-1}}\int_{0}^{x}s^{i-2}\  ds -\frac{i-1}{x^{i-1}}\int_{\frac{x}{2}}^{x}(2s-x)^{i-2}\  ds\\
&=1-\frac{i-1}{x^{i-1}}\int_{\frac{x}{2}}^{x}(2s-x)^{i-2}\ ds = \frac{1}{2}.
\end{align*}
and similarly
$\P(X_{(i+2)} >2X_{(i+1)} - x|X_{(i)}=x) = \frac{1}{2}.$
Substituting this in \eqref{eq:zero_q}, we obtain
$$\P(E_i^+ \cap E_i^-) = \int_{0}^{1}\frac{f_{X_{(i)}}(x)}{4}dx = \frac{1}{4}.$$
Nodes $1$ and $N$ have no incoming edges if the event $E_1^+$ and $E_N^-$ are true respectively. Moreover, nodes $2$ and $N-1$ always have an incoming edge from nodes $1$ and $N$ respectively. Thus, $$\E \bigg[\frac{Q^N_{1,1,0}}{N}\bigg] = \frac{1}{N} \E\left[\1\{E_1^+\}+\1\{E_{N}^-\}+\sum_{i=3}^{N-2} \1\{E_i^+ \cap E_i^-\}\right]=\frac{1}{N}\bigg[2\cdot\frac{1}{2}+(N-4)\cdot\frac{1}{4}\bigg].$$
For $N\ge 4$, we obtain $\E \Big[\frac{Q^N_{1,1,0}}{N}\Big] = \frac{1}{4}$. Similar computations for $\E \Big[\frac{Q^N_{1,1,1}}{N}\Big]$ and $\E \Big[\frac{Q^N_{1,1,2}}{N}\Big]$
prove the proposition.
\end{proof}

In fact, a stronger statement to Proposition~\ref{prop:frac_indeg} holds for $Q^N_{1,1,0}$ and $Q^N_{1,1,2}$ which is stated below.

\begin{prop}
The number of vertices with in-degree $0$ and $2$ in the $1$-NN graph on $[0,1]$ 
are the same, that is, $Q^N_{1,1,0} = Q^N_{1,1,2}$ a.s.
\label{prop:r0r2}
\end{prop}

\begin{proof}
A crucial observation for the $1$-NN graph is that the only possible cycles are $2$-cycles comprising of two nodes that are mutual nearest neighbours (see \citep{eppstein1997nearest}). If all the components of the $1$-NN graph $G_1^N$ in one dimension are $2$-cycles as shown in Fig. \ref{fig:r0r2}(a), then $Q^N_{1,1,0} = Q^N_{1,1,2} = 0$. On the other hand, every component of more than two vertices contains exactly one $2$-cycle. Such components can be of two types: starting at a node of in-degree $0$ and culminating in the $2$-cycle as shown in Fig. \ref{fig:r0r2}(b) or starting and ending at a node of in-degree $0$ with a $2$-cycle in between as shown in Fig. \ref{fig:r0r2}(c). In the former case, since the starting vertex has in-degree $0$ and one of the vertices in the $2$-cycle associated with the component has in-degree $2$, every component of size greater than $2$ can be associated with such a unique pair of vertices with in-degrees of $0$ and $2$. Likewise, in the latter case, the component has two node in the $2$-cycle which have in-degree $2$ and the two nodes at the extremes of the component have in-degree $0$. Thus, in all the cases we have that $Q^N_{1,1,0} = Q^N_{1,1,2}$ almost surely which proves the claim.
\end{proof}

\begin{figure}
\centering
\subfloat[Case 1: $Q^N_{1,1,0}=Q^N_{1,1,2} = 0$]{
\includegraphics[width=0.5\linewidth]{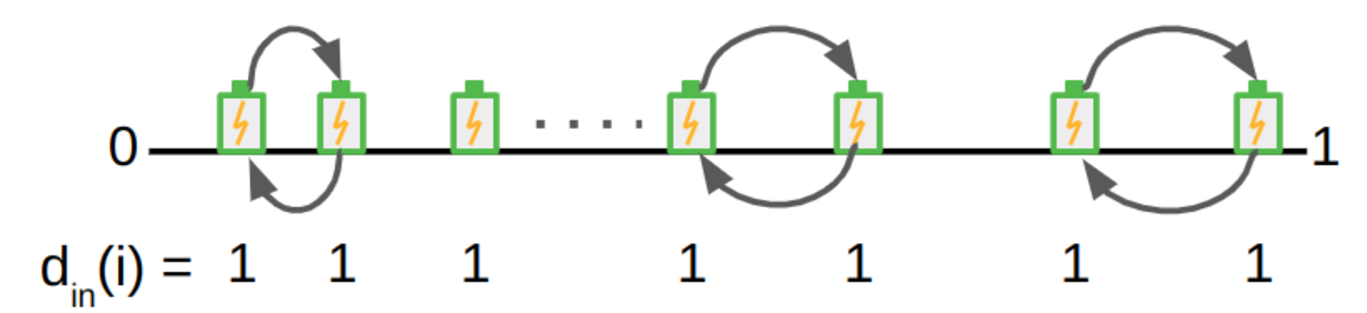}
\label{fig:R00}}\\
\subfloat[Case 2: $Q^N_{1,1,0}=Q^N_{1,1,2} = 1$]{
\includegraphics[width=0.5\linewidth]{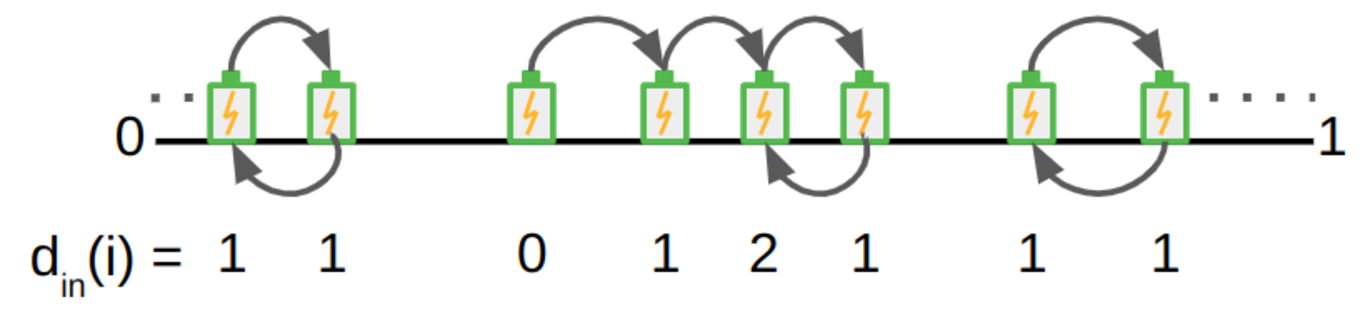}}
\subfloat[Case 3: $Q^N_{1,1,0}=Q^N_{1,1,2} = 2$]{
\includegraphics[width=0.5\linewidth]{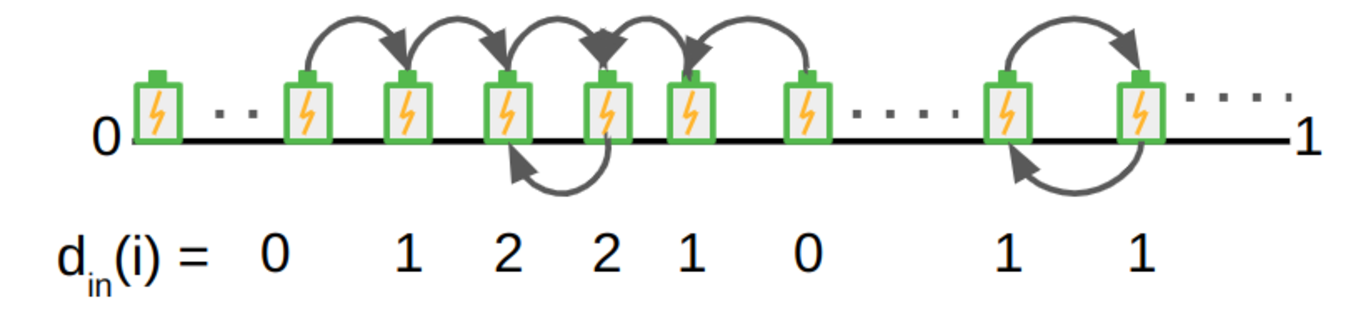}}
\caption{Illustration of $Q^N_{1,1,0}=Q^N_{1,1,2}$.  An arrow $i \to j$ indicates $j$ is the nearest neighbour of $i$.}
\label{fig:r0r2}
\end{figure}

\subsection{$(1,p)$-NNS strategy}
\label{sub:nns}
In this section, we characterize the overloaded servers for the $(1,p)$-NNS strategy in one-dimension by proving Theorem \ref{thm:main1d}.

\begin{proof}[Proof of Corollary \ref{thm:convergence_1d}]
Let $d=1$ and $k=1$. Then $\alpha_1 = 2$,
and by Theorem~\ref{thm:main2d} we find that
$\cO_N=0$ a.s.\ when $\frac{\lambda}{\mu} \le \frac{1}{1+p}$.

Assume next that $\frac{1}{1+p} < \frac{\lambda}{\mu} \le 1$.
Then the constant $\theta$ in Theorem~\ref{thm:main2d} equals
$\theta = 1 + \frac{1}{p} \left( \frac{\mu}{\lambda} -1 \right)$
and is bounded by $1 \le \theta < 2$.
Theorem~\ref{thm:main2d} then implies that
\[
 \cO_N \asto
 q_{1,1,2}
 \quad\text{and}\quad
 \sqrt{N} \big(\cO_N - \E \cO_N \big) \dto \cN(0,\sigma_{1,1}^2).
\]
The claim follows by plugging in the values of the constants
$q_{1,1,2} = \frac14$ and $\sigma_{1,1}^2 = \frac{19}{240}$ obtained in \citep{bahadir2016number}.
\end{proof}

\begin{proof}[Proof of Theorem~\ref{thm:main1d}]
Because $\frac{1}{1+p} < \frac{\lambda}{\mu} \le 1$,
we see that the constant $\theta$ in \eqref{eq:OverloadFractionSum}
is bounded by $1 \le \theta < 2$.
Because $\alpha_1 = 2$ for $d=1$, we see that \eqref{eq:OverloadFractionSum} 
takes the form $\cO_N = \frac{1}{N} Q^N_{1,1,2}$.
Proposition~\ref{prop:frac_indeg} now yields the result $\E[\cO_N] = \frac{1}{4}$.
\end{proof}

\begin{remark}
While Theorem \ref{thm:main1d} characterizes the fraction of overloaded servers for $N \ge 4$, it can be inferred exhaustively that $\E \cO_2 =0$ and $\E\cO_3 = \frac{1}{3}$.
\end{remark}

\subsection{Left-right strategy}
\label{sec:other_1d}

\begin{proof}[Proof of Theorem \ref{thm:lrnns}]
Consider a service station $i$ that is not located at the boundary of $[0,1]$.
Customers arriving at queue~$i$ join queue $i$ with probability $1-\ell-r$.
Customers arriving to the nearest queue to the left (resp.\ right) of $i$
join queue $i$ with probability $\ell$ (resp.\ right). Therefore, the
net rate of customers joining queue~$i$ equals
\[
 \leff(i)
 \ = \ \lambda(1-\ell-r) + \lambda \ell + \lambda r
 \ = \ \lambda.
\]
A similar computation shows that $\leff(i) = \lambda$ also for the
service stations located at the boundary of $[0,1]$. Therefore,
$\cO_N = 0$ a.s.
\end{proof}

\section{Numerical results}
\label{sec:simu}
In this section, we present simulation experiments complementing our theoretical results. We restrict to the case of $k=1$ since the constants in our results (or their approximation using Monte Carlo simulations) are known only for the $(1,p)$-NNS strategy. 

Consider $N=1000$ servers deployed uniformly in $[0,1]$, each of which is associated with an independent arrival process of intensity $\lambda = 1$. The arrival process is simulated for $t=1000$ time units with a shift probability $p$. The probability parameter $p$ is chosen from the set $\{1,0.75,0.5,0.25\}$. Fig.~\ref{fig:simu1d} plots the histogram of the observed number of overloaded servers ($Q^N_{1,1,2}$) and the servers with no change in load ($Q^N_{1,1,1}$) over $1000$ instantiations of the server locations in the one dimensional case. The number of overloaded servers is concentrated well around the mean of $N/4 = 250$ irrespective of the probability $p$ corroborating with Theorem \ref{thm:main1d}.

Fig. \ref{fig:simu2d} shows a similar histogram of the expected fraction of overloaded servers when the servers are distributed in $[0,1]^2$. The values of the constants $q_{2,1,j}$ for $j=0,1,\cdots,5$ have been computed using Monte Carlo simulations in \citep{cuzick1990spatial} which is reproduced in Table \ref{tab:values}.

\begin{table}
\centering
\small
\begin{tabular}{rc}
\toprule
$j$ & $q_j(2,1)$ \\
\midrule
$0$ & $2.84 \times 10^{-1}$ \\
$1$ & $4.63 \times 10^{-1}$ \\
$2$ & $2.22 \times 10^{-1}$ \\
$3$ & $3.04 \times 10^{-2}$ \\
$4$ & $6.56 \times 10^{-4}$ \\
$5$ & $1.90 \times 10^{-7}$ \\
\bottomrule
\end{tabular}
\caption{
Values of the asymptotic $k$-NN in-degree distribution $q_{d,k,j}$ for $d=2$ and $k=1$.
Recall that $q_{2,1,j}=0$ for $j \ge 6$.}
\label{tab:values}
\end{table}

When the servers are distributed in $2$D and customers follow the $(1,p)$-NNS strategy, Theorem \ref{thm:main2d} states that the fraction of overloaded servers converges to $ \sum_{j=2}^{5} q_{2,1,j}$ almost surely as $ N\to \infty$. Using the values from Table \ref{tab:values}, we obtain $\sum_{j=2}^{5} q_{2,1,j} \approx 0.252$ which is where the blue peak is located in Fig. \ref{fig:simu2d}. Further supporting the result is the peak observed of the fraction of servers with no change in load which is equal to $q_{2,1,1} = 0.463$.
\begin{figure}
\centering
\subfloat[$d=1$]{\includegraphics[width=0.5\linewidth]{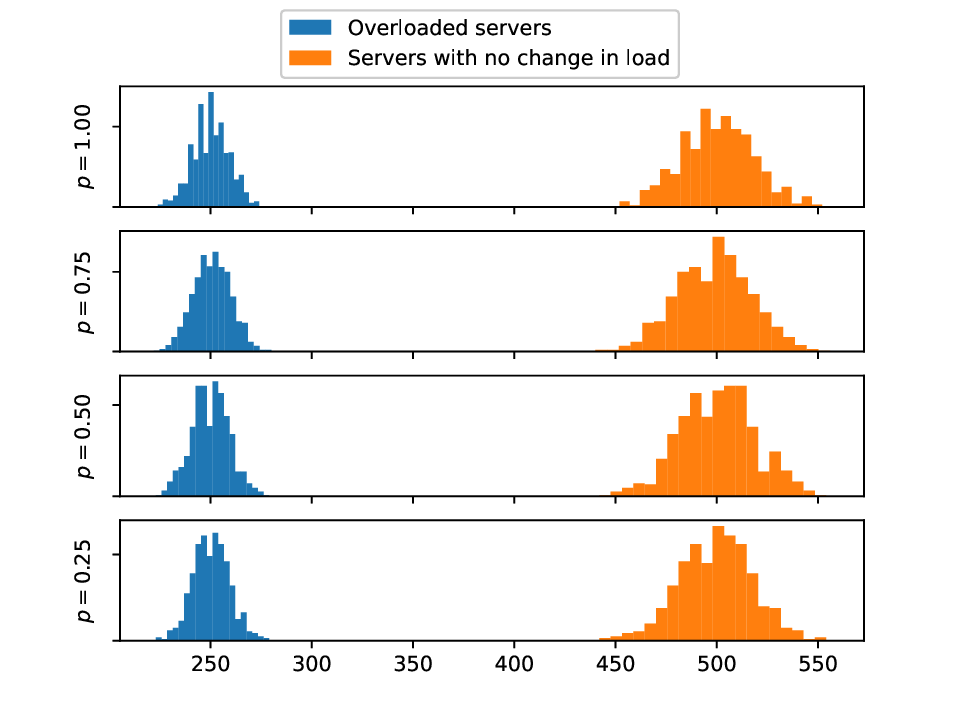}
\label{fig:simu1d}
}
\subfloat[$d=2$]{\includegraphics[width=0.5\linewidth]{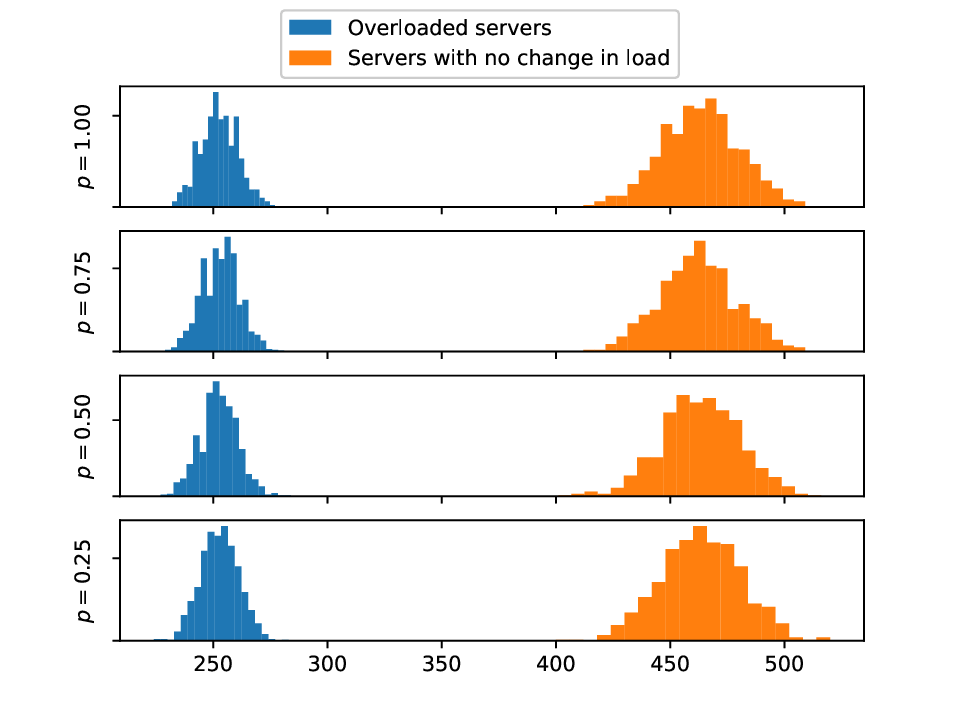}
\label{fig:simu2d}
}
\caption{Histogram of the number of overloaded servers and the servers with no change in load for different probability values of the $(1,p)$-NNS strategy in (a) one dimension and (b) two dimension.}
\end{figure}

\section{Conclusions and future work}
\label{sec:conc_fw}
In this paper, we considered multiple servers on a $d$-dimensional Euclidean space, where arriving customers follow a probabilistic policy to either remain in the queue or to join a nearest neighbour.  In this setting, we provided distributional properties of the fraction of overloaded servers in the stationary regime. Numerous questions remain yet to be answered in this setting, some of which are discussed below.

\subsection{Spatial distribution of overloaded servers.}

In the long run, overloaded servers require higher maintenance and allocating
more capacity.
This necessitates a characterization of their distribution in space. Fig. \ref{fig:spatial1d} and Fig. \ref{fig:spatial2d} show the distribution of the overloaded servers in space in one and two dimensions respectively. It is reasonable to expect that overloaded servers do not occur close to one another. For example, three adjacent servers in dimension one cannot be overloaded since their total in-degree is upper bounded by $5$ in $G_{1,1}^N$.
This suggests that the spatial point pattern of overloaded servers
resembles more a Mat\'ern point process than a standard Poisson point process \citep{Kuronen_Leskela_2013}.
The spatial distribution is of particular importance in the context of electric vehicles since it gives the distribution of locations of charging stations requiring maintenance.

\begin{figure}
\subfloat[$p=1$]{
\includegraphics[width=0.45\linewidth]{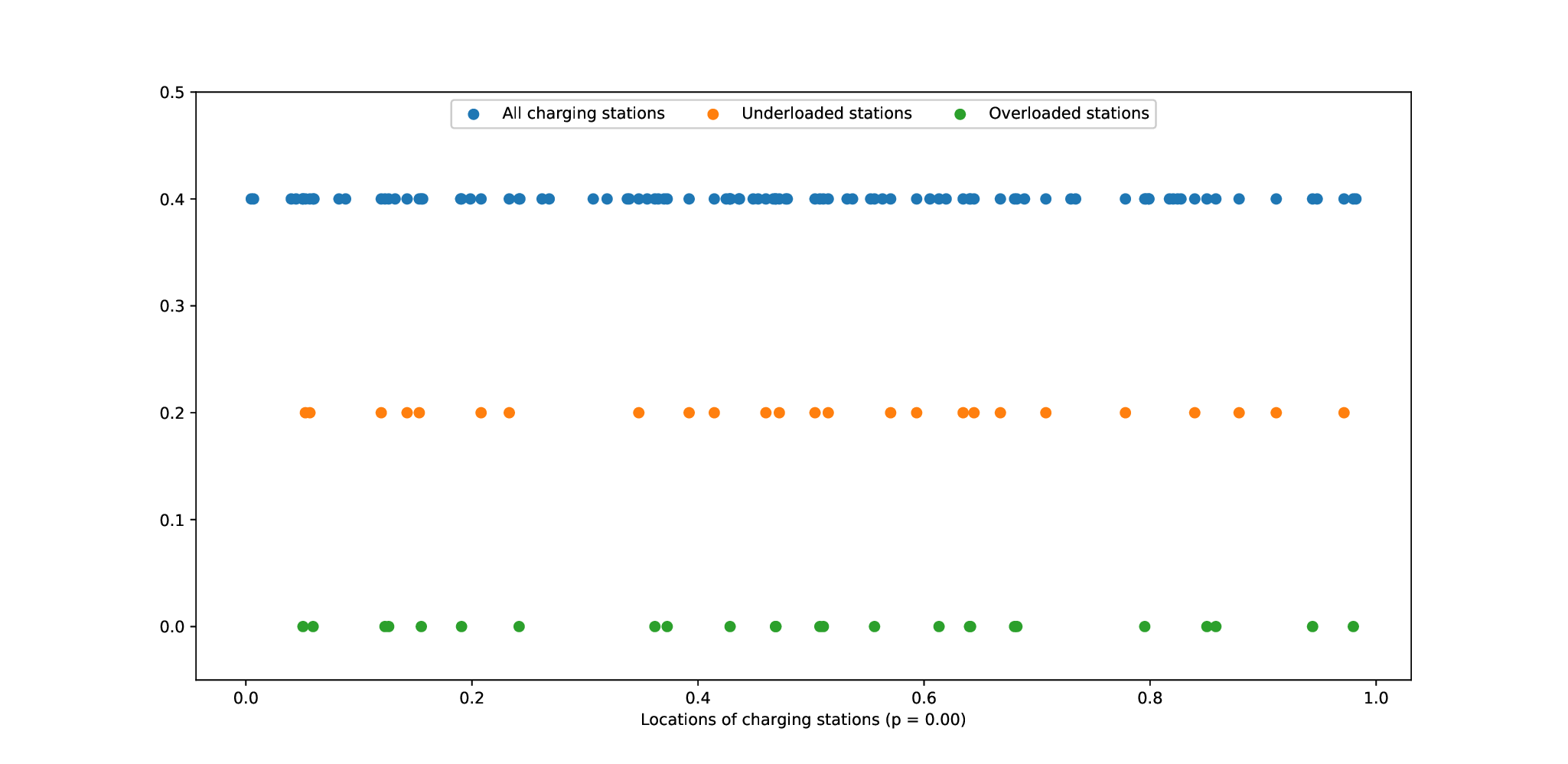}}
\subfloat[$p=0.75$]{
\includegraphics[width=0.45\linewidth]{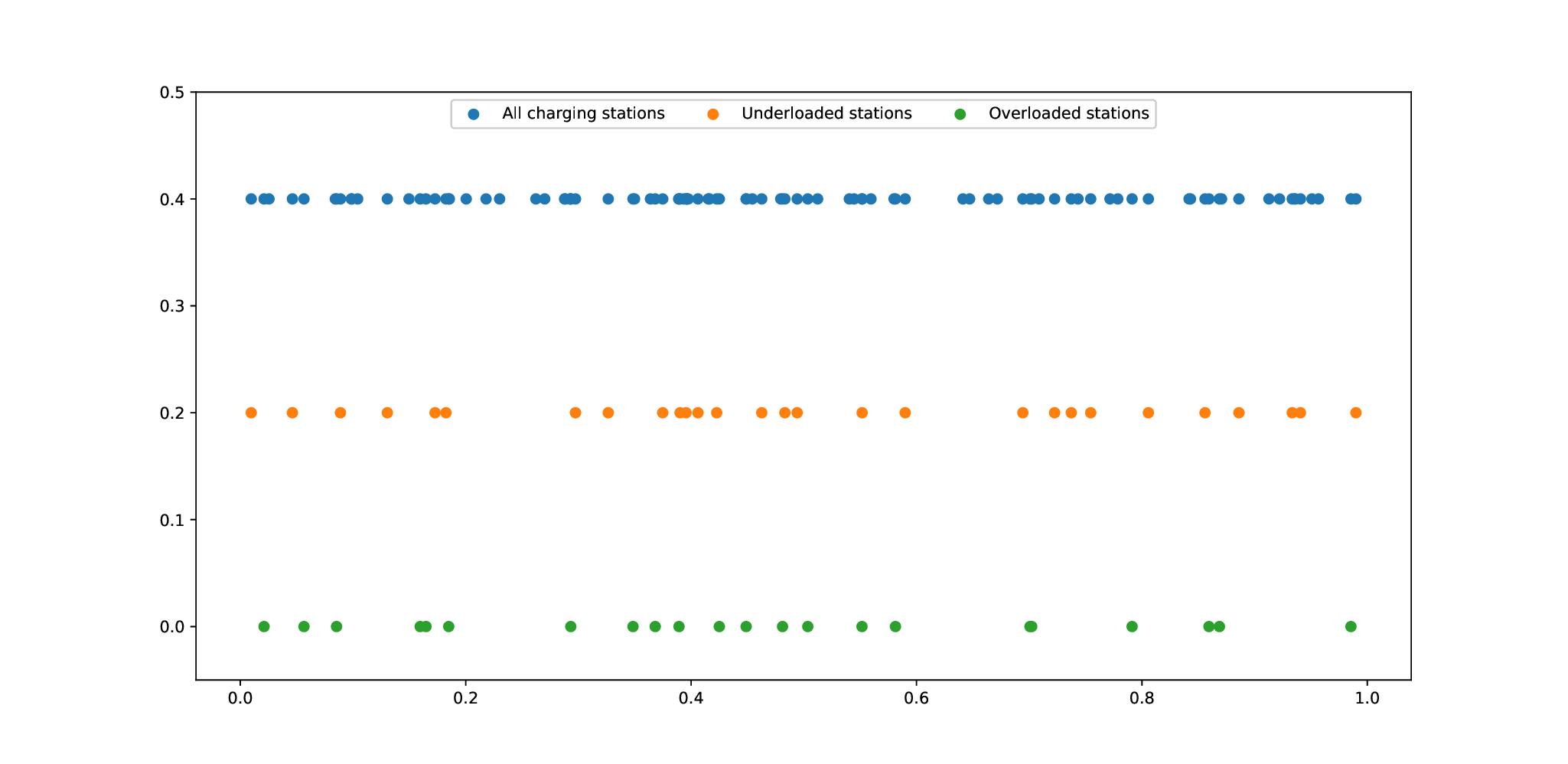}}\\
\subfloat[$p=0.60$]{
\includegraphics[width=0.45\linewidth]{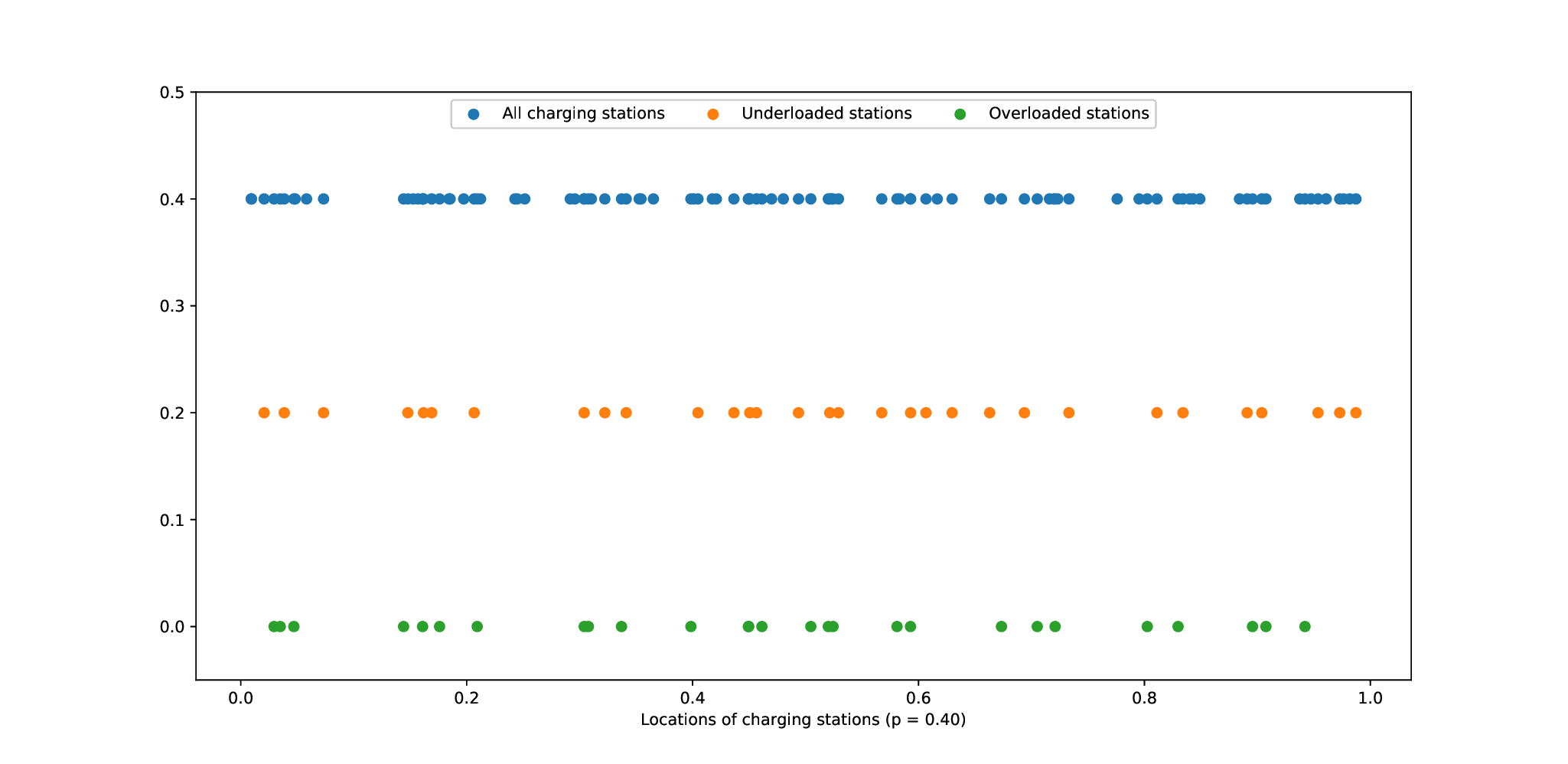}}
\subfloat[$p=0.50$]{
\includegraphics[width=0.45\linewidth]{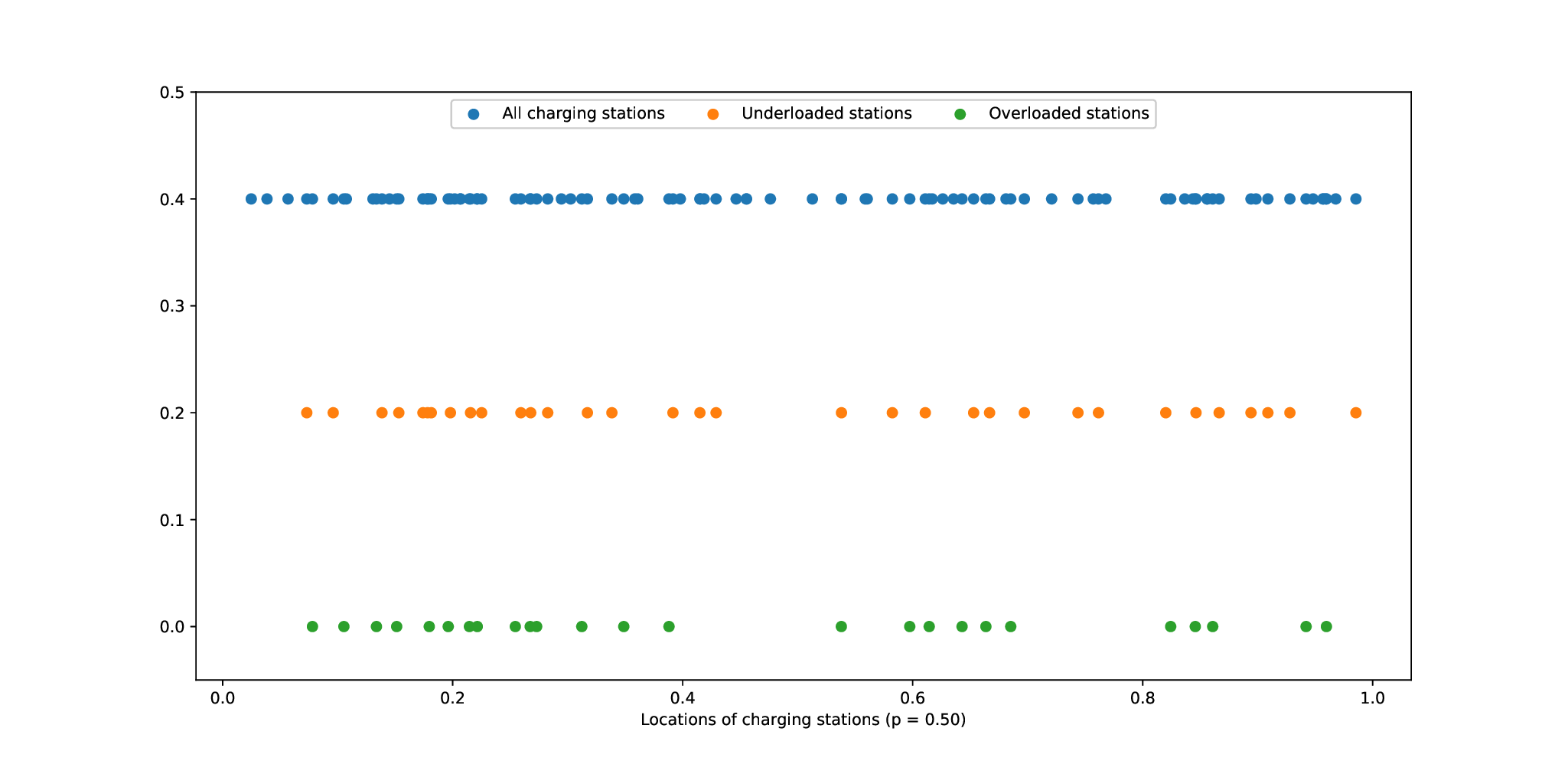}}
\caption{Spatial distribution of $100$ charging stations (top blue), underloaded stations (middle orange) and overloaded stations (bottom green) in $[0,1]$ for different probability values.}
\label{fig:spatial1d}
\end{figure}
\begin{figure}
\subfloat[$p=1$]{
\includegraphics[width=0.46\linewidth]{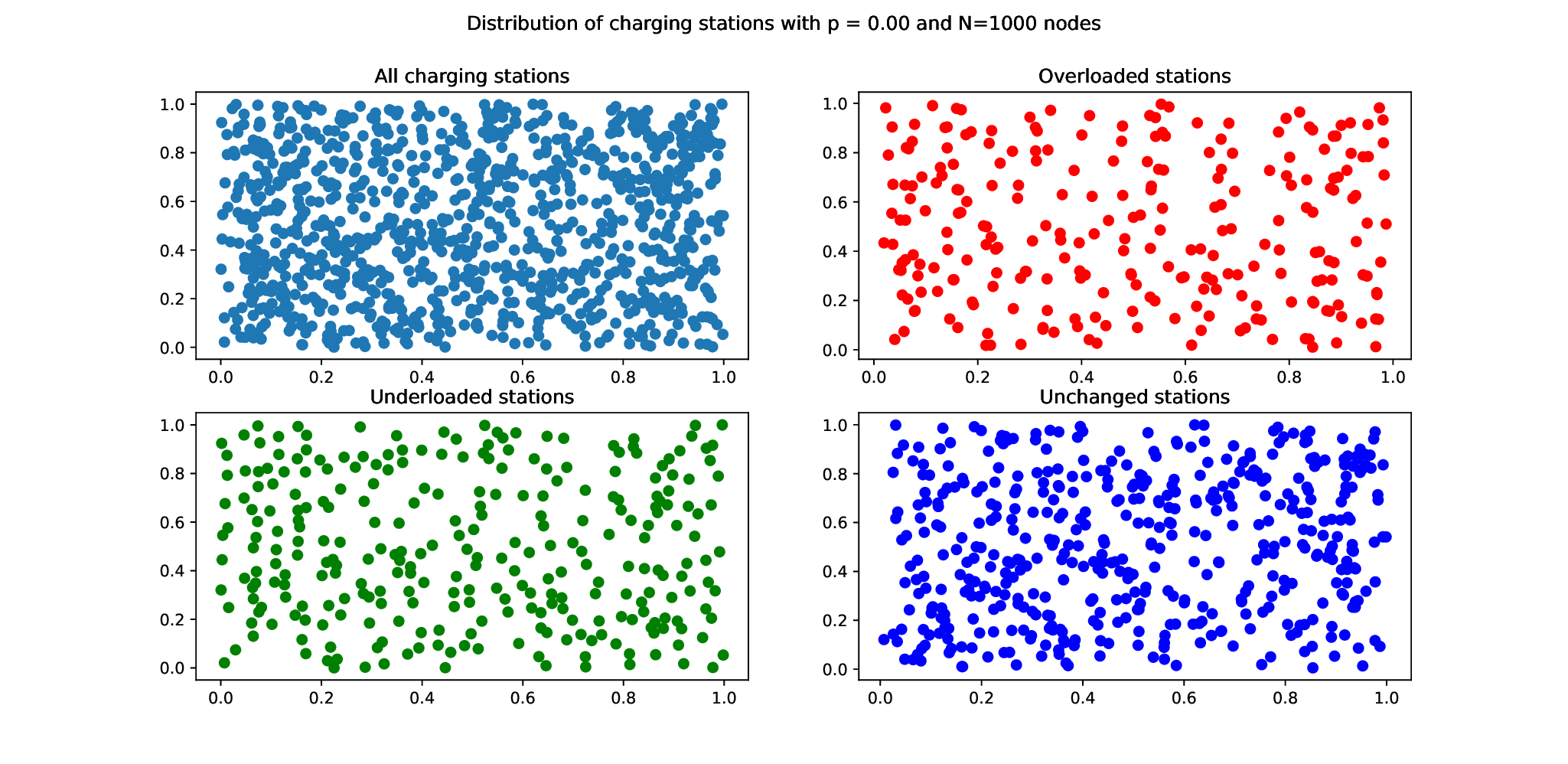}}
\subfloat[$p=0.75$]{
\includegraphics[width=0.46\linewidth]{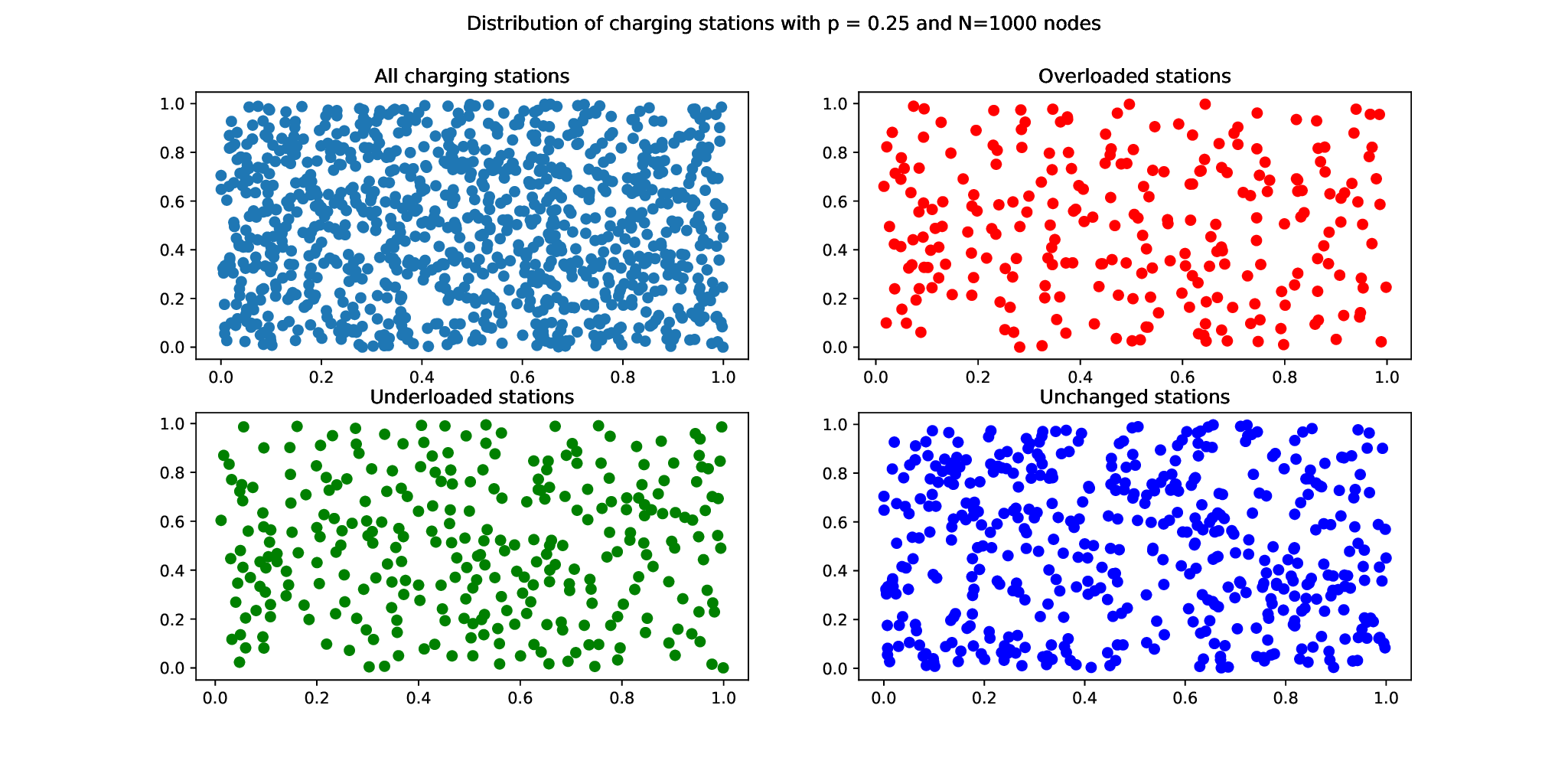}}\\
\subfloat[$p=0.60$]{
\includegraphics[width=0.46\linewidth]{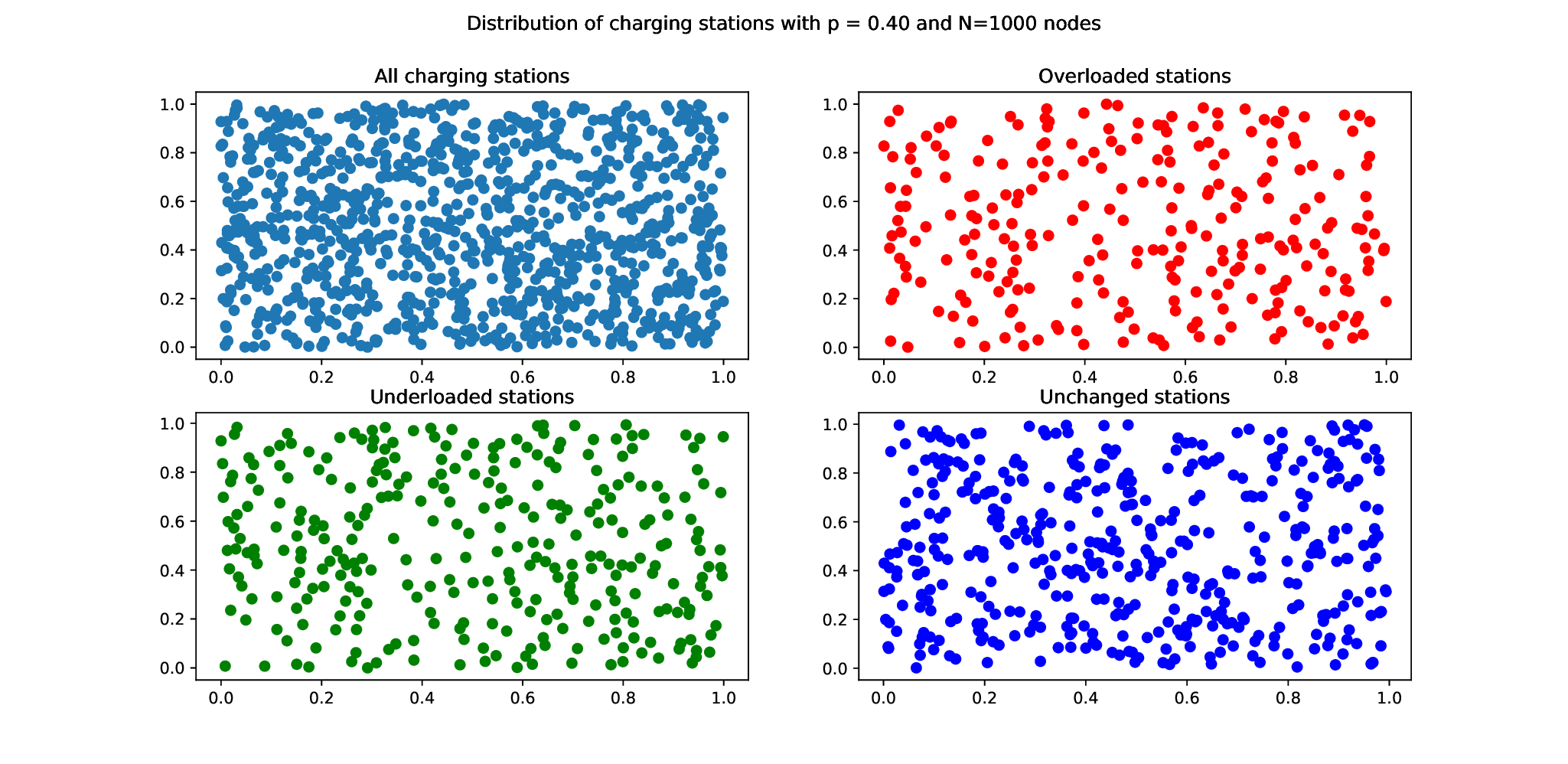}}
\subfloat[$p=0.50$]{
\includegraphics[width=0.46\linewidth]{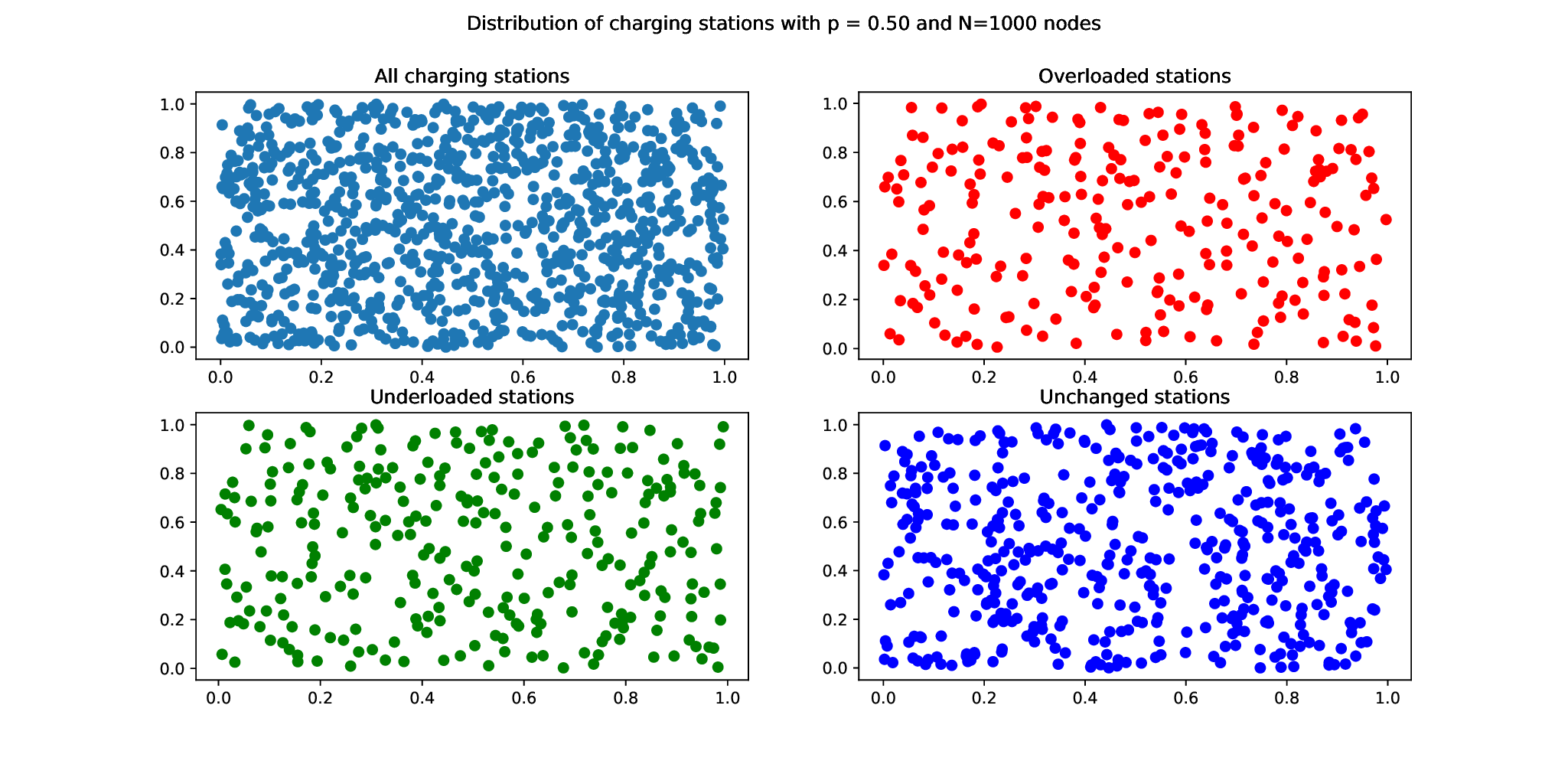}}
\caption{Spatial distribution of $1000$ charging stations (top-left), overloaded stations (top-right), underloaded stations (bottom-left) and stations with unchanged load (bottom-right) in $[0,1]^2$ for different probability values.}
\label{fig:spatial2d}
\end{figure}

\subsection{State-dependent shift probabilities and adaptive mobility strategies}

Exploring state-dependent strategies, where shift probabilities change with transient parameters such as the queue length, is an important research direction. On a related note, \citep{Leskela_Unger_2012,rolla2017stability} investigate the steady-state behaviour and system stability when the arrivals are distributed on the circle and a myopic server operates according to a greedy policy serving the nearest arrivals. In our case, however, we have multiple servers. When customers change queues based on the queue length they observe when they arrive, characterization of the steady-state behaviour and stability are compelling future directions.
As a first step in this direction, one could study an adaptive $(k,p)$-NNS strategy where
customers join empty queues or, with probability $p$, join a non-empty queue; otherwise, they select one of the $k$ nearest neighboring queues.  Such strategies likely improve server utilization and reduce idling.  Techniques similar to those in \citep{Borst_Jonckheere_Leskela_2008} might be feasible for showing that overloaded queues under adaptive strategies would also be overloaded under non-adaptive ones, potentially using non-adaptive results as upper bounds for adaptive cases. Understanding these dynamics under state-dependent and adaptive strategies remains an interesting area for future research.

\subsection{Impact of ambient space.}

How does the ambient space impact $\cO_N$? If factors from the Euclidean space such as traversal times to other queues are considered, several classical queueing problems (delay, wait times etc.) can be formulated which might be interesting in their own right, see for example \citep{Pender_2022} and references therein.

\bibliographystyle{abbrv}
\bibliography{references}
\end{document}

%% file: my_header.tex
\usepackage[utf8]{inputenc} 
\usepackage[T1]{fontenc}
\usepackage{url}
\usepackage{ifthen}
\usepackage{mdframed}
\usepackage{cite}
\usepackage{enumerate}
\usepackage[cmex10]{amsmath}
\usepackage{graphicx}
\usepackage{subfig}
\usepackage{dsfont}
\usepackage{mathtools}
\usepackage{enumitem}
\usepackage{bm}
\usepackage{cite}
\usepackage{orcidlink}
\usepackage{xcolor}
\usepackage{soul}
\usepackage{booktabs}
\usepackage{comment}

\DeclarePairedDelimiter\floor{\lfloor}{\rfloor}

\usepackage{amssymb, amsmath, amsthm, amsfonts}
\usepackage{xcolor}

\newtheorem{thm}{Theorem}[section]

\newtheorem{prop}[thm]{Proposition}

\newtheorem{assumption}{Assumption}

\newcommand{\din}{d_{\rm in}}
\newcommand{\dout}{d_{\rm out}}

\def\0{\mathbf{0}}

\def\1{\mathbf{1}}

\def\E{\mathbb{E}}
\def\P{\mathbb{P}}
\def\R{\mathbb{R}}

\def\cN{\mathcal{N}}
\def\cO{\mathcal{O}}

\def\leff{\lambda_{\text{eff}}}

\def\weq{ \ = \ }

\theoremstyle{definition}
\newtheorem{defn}{Definition}[section]
\newtheorem{remark}{Remark}[section]
\newtheorem{corr}{Corollary}[section]

\definecolor{brickred}{cmyk}{0,0.89,0.94,0.28}
\definecolor{goldenrod}{cmyk}{0,0.10,0.84,0}
\definecolor{purple}{cmyk}{0.45,0.86,0,0}
\definecolor{rawsienna}{cmyk}{0,0.72,1,0.45}
\definecolor{olivegreen}{cmyk}{0.64,0,0.95,0.40}
\definecolor{peach}{cmyk}{0,0.5,0.7,0}
\definecolor{darkolive}{rgb}{0.,0.4,0.}
\definecolor{vibrantgreen}{cmyk}{0.75,0,0.95,0.25}
\definecolor{darkblue}{rgb}{0.2,0.4,0.95}

\definecolor{PastelYellow}{HTML}{fdffb4}

\newcommand{\remove}[1]{}

\newcommand{\dto}{\stackrel{\rm{d}}{\rightarrow}}
\newcommand{\asto}{\stackrel{\rm{a.s.}}{\rightarrow}}

\sethlcolor{PastelYellow}

%% file: arxiv_header.tex
\usepackage[margin=0.75in,footskip=0.25in]{geometry}
\usepackage[utf8]{inputenc} 
\usepackage[T1]{fontenc}
\usepackage{url}
\usepackage{ifthen}
\usepackage{mdframed}
\usepackage{cite}
\usepackage{enumerate}
\usepackage[cmex10]{amsmath}
\usepackage{graphicx}
\usepackage{subfig}
\usepackage{dsfont}
\usepackage{mathtools}
\usepackage{enumitem}
\usepackage{bm}
\usepackage{cite}
\usepackage{authblk}

\usepackage{amssymb, amsmath, amsthm, amsfonts}
\usepackage{xcolor}

\def\0{\mathbf{0}}

\def\1{\mathbf{1}}

\def\E{\mathbb{E}}
\def\P{\mathbb{P}}
\def\R{\mathbb{R}}

\def\cN{\mathcal{N}}
\def\cO{\mathcal{O}}
\def\weq{ \ = \ }

\theoremstyle{definition}

\newcommand{\citep}[1]{\cite{#1}}

%% file: arxiv_author.tex
\title{Spatial Queues with Nearest Neighbour Shifts \footnote{A part of this work was accepted to the conference International Teletraffic Congress (ITC 35) held between 3--5 October 2023 in Turin, Italy.}}
\author[1]{B. R. Vinay Kumar*\orcidlink{0000-0002-7329-8659}\thanks{This work was done when VK was a post-doctoral researcher at INRIA, Sophia Antipolis, France and when he was visiting Aalto University, Finland in August 2023. VK was supported by the French government through the RISE Academy of UCA$^\text{JEDI}$ Investments in Future project managed by the National Research Agency (ANR) with reference number ANR-15-IDEX-0001, and the Aalto Science Institute (AScI) visiting researcher fellowship.} }
\affil[1]{Eindhoven University of Technology, P.O. Box 513, 5600 MB Eindhoven, The Netherlands. \authorcr E-mail: v.k.bindiganavile.ramadas@tue.nl}
\author[2]{Lasse Leskel\"{a}\orcidlink{0000-0001-8411-8329}} 
\affil[2]{Department of Mathematics and Systems Analysis, Aalto University, Otakaari 1, 02150 Espoo, Finland. E-mail: lasse.leskela@aalto.fi}

\bibliographystyle{abbrv}

%% file: PEVA_arxiv_revision1.bbl
\begin{thebibliography}{10}

\bibitem{altman_queueing_1994}
E.~Altman and H.~Levy.
\newblock Queueing in space.
\newblock {\em Advances in Applied Probability}, 26(4):1095--1116, 1994.

\bibitem{atat_stochastic_2020}
R.~Atat, M.~Ismail, and E.~Serpedin.
\newblock Stochastic geometry planning of electric vehicles charging stations.
\newblock In {\em Proc. {International} {Conference} on {Acoustics}, {Speech}
  and {Signal} {Processing} ({ICASSP})}, pages 3062--3066, May 2020.

\bibitem{bahadir2016number}
S.~Bahad{\i}r and E.~Ceyhan.
\newblock On the number of reflexive and shared nearest neighbor pairs in
  one-dimensional uniform data.
\newblock {\em arXiv preprint arXiv:1605.01940}, 2016.

\bibitem{bahadirNumberWeaklyConnected2021}
S.~Bahadır and E.~Ceyhan.
\newblock On the number of weakly connected subdigraphs in random {kNN
  }digraphs.
\newblock {\em Discrete \& Computational Geometry}, 65(1):116--142, 2021.

\bibitem{Borst_Jonckheere_Leskela_2008}
S.~Borst, M.~Jonckheere, and L.~Leskel\"a.
\newblock Stability of parallel queueing systems with coupled service rates.
\newblock {\em Discrete Event Dynamic Systems}, 18(4):447--472, 2008.

\bibitem{cuzick1990spatial}
J.~Cuzick and R.~Edwards.
\newblock Spatial clustering for inhomogeneous populations.
\newblock {\em Journal of the Royal Statistical Society Series B: Statistical
  Methodology}, 52(1):73--96, 1990.

\bibitem{david2004order}
H.~A. David and H.~N. Nagaraja.
\newblock {\em Order statistics}.
\newblock John Wiley \& Sons, 2004.

\bibitem{dong_electric_2019}
G.~Dong, J.~Ma, R.~Wei, and J.~Haycox.
\newblock Electric vehicle charging point placement optimisation by exploiting
  spatial statistics and maximal coverage location models.
\newblock {\em Transportation Research Part D: Transport and Environment},
  67:77--88, Feb. 2019.

\bibitem{dudin_analysis_2023}
S.~A. Dudin, O.~S. Dudina, and O.~I. Kostyukova.
\newblock Analysis of a queuing system with possibility of waiting customers
  jockeying between two groups of servers.
\newblock {\em Mathematics}, 11(6):1475, Jan. 2023.

\bibitem{eppstein1997nearest}
D.~Eppstein, M.~S. Paterson, and F.~F. Yao.
\newblock On nearest-neighbor graphs.
\newblock {\em Discrete \& Computational Geometry}, 17:263--282, 1997.

\bibitem{kavitha_queuing_2009}
V.~Kavitha and E.~Altman.
\newblock Queuing in space: {Design} of message ferry routes in static ad hoc
  networks.
\newblock In {\em 2009 21st {International} {Teletraffic} {Congress}}, pages
  1--8, Sept. 2009.

\bibitem{koenigsberg_jockeying_1966}
E.~Koenigsberg.
\newblock On jockeying in queues.
\newblock {\em Management Science}, Jan. 1966.

\bibitem{Kuronen_Leskela_2013}
M.~Kuronen and L.~Leskel\"a.
\newblock Hard-core thinnings of germ--grain models with power-law grain sizes.
\newblock {\em Advances in Applied Probability}, 45(3):595--625, 2013.

\bibitem{Leskela_Unger_2012}
L.~Leskel{\"a} and F.~Unger.
\newblock Stability of a spatial polling system with greedy myopic service.
\newblock {\em Annals of Operations Research}, 198(1):165--183, 2012.

\bibitem{lin_optimal_2022}
B.~Lin, Y.~Lin, and R.~Bhatnagar.
\newblock Optimal policy for controlling two-server queueing systems with
  jockeying.
\newblock {\em Journal of Systems Engineering and Electronics}, 33(1):144--155,
  Feb. 2022.

\bibitem{newman1983nearest}
C.~M. Newman, Y.~Rinott, and A.~Tversky.
\newblock Nearest neighbors and {Voronoi} regions in certain point processes.
\newblock {\em Advances in Applied Probability}, 15(4):726--751, 1983.

\bibitem{panigrahy2022analysis}
N.~K. Panigrahy, T.~Vasantam, P.~Basu, D.~Towsley, A.~Swami, and K.~K. Leung.
\newblock On the analysis and evaluation of proximity-based load-balancing
  policies.
\newblock {\em ACM Transactions on Modeling and Performance Evaluation of
  Computing Systems}, 7(2-4):1--27, 2022.

\bibitem{Pender_2022}
J.~Pender.
\newblock {k-Nearest} neighbor queues with delayed information.
\newblock {\em International Journal of Bifurcation and Chaos}, 32(12):2250174,
  2022.

\bibitem{ren_novel_2020}
C.~Ren and Y.~Hou.
\newblock {\em A novel evaluation method of electric vehicles charging network
  based on stochastic geometry}.
\newblock Oct. 2020.
\newblock Pages: 468.

\bibitem{rolla2017stability}
L.~T. Rolla and V.~Sidoravicius.
\newblock Stability of the greedy algorithm on the circle.
\newblock {\em Communications on Pure and Applied Mathematics},
  70(10):1961--1986, 2017.

\bibitem{ross2010first}
S.~Ross.
\newblock {\em A first course in probability}.
\newblock Pearson, 2010.

\bibitem{stacey_greedy_2011}
K.~W. Stacey and D.~P. Kroese.
\newblock Greedy servers on a torus.
\newblock In {\em Proceedings of the 2011 {Winter} {Simulation} {Conference}
  ({WSC})}, pages 369--380, Dec. 2011.

\bibitem{celik_dynamic_2010}
G.~D. Çelik and E.~Modiano.
\newblock Dynamic vehicle routing for data gathering in wireless networks.
\newblock In {\em 49th {IEEE} {Conference} on {Decision} and {Control}
  ({CDC})}, pages 2372--2377, Dec. 2010.

\end{thebibliography}
